\documentclass[11pt]{article}
\usepackage{amsmath}
\usepackage{amsfonts}
\usepackage{amssymb}
\usepackage{epsfig}
\usepackage[font=normalsize]{subfig}
\usepackage{amscd}
\usepackage{latexsym}
\usepackage[colorlinks=black,linkcolor=red,anchorcolor=blue,citecolor=green]{hyperref}
\usepackage{CJK}
\usepackage{color}
\usepackage{leftidx}
\usepackage{multirow}
\usepackage{booktabs}
\usepackage{threeparttable}

\newtheorem{theorem}{\bf Theorem}[section]
\newtheorem{lemma}[theorem]{\bf Lemma}

\newtheorem{example}[theorem]{\bf Example}

\newtheorem{algorithm}[theorem]{\bf Algorithm}

\newenvironment{proof}{\noindent{\em Proof:}}{\quad \hfill$\Box$\vspace{2ex}}

\DeclareMathOperator{\diag}{diag}
\DeclareMathOperator{\Log}{Log}

\def \aN {\mathbb N}

\def \bL {\mathcal{L}}

\def \bO {\mathcal{O}}

\def \bD {\mathcal{D}}

\def \Be {\textbf{e}}
\def \Bf {\textbf{f}}

\def \Bx {\textbf{x}}

\def \Bq {\textbf{q}}
\def \Bh {\textbf{h}}

\newcommand{\Rnum}[1]{\uppercase\expandafter{\romannumeral #1\relax}}
\newcommand{\rnum}[1]{\lowercase\expandafter{\romannumeral #1\relax}}
\begin{document}
\begin{CJK}{GBK}{song}

\title{\bf A Levin method for logarithmically singular oscillatory integrals }%Study Report 8
\author{ Yinkun Wang
         \footnotemark[1],
         Shuhuang Xiang
         \footnotemark[2]
         \footnotemark[3]
         }

\renewcommand{\thefootnote}{\fnsymbol{footnote}}

\footnotetext[1]{Department of Mathematics, National University of Defense Technology, Changsha, P. R. China.}
\footnotetext[2]{School of Mathematics and Statistics, Central South University, Changsha, Hunan, P. R. China.}
\footnotetext[3]{Correspondence author: xiangsh@mail.csu.edu.cn }
\date{}
\maketitle{}
\begin{abstract}
We propose a new stable Levin method to compute oscillatory integrals with logarithmic singularities and without stationary points.
To avoid the singularity, we apply the technique of singularity separation and transform the singular ODE into two non-singular ODEs, which can be solved efficiently by the collocation method.
Applying the equivalency of the new Levin method for the singular oscillatory integrals and  the Filon method when the oscillator is linear, we consider the convergence  of the new Levin method. This new method shares the proposition that
 less error for higher oscillation. Several numerical experiments are presented to validate the efficiency of the proposed method.
\end{abstract}


\section{Introduction}
We consider in this paper highly oscillatory integrals of Fourier type of the form
\begin{equation}
    I_{\log,w}^{[0,a]}[f,g]:=\int_0^af(x)\log xe^{iwg(x)}dx,
\end{equation}
where $a$ is a positive real number, $f$ and $g$ are suitably smooth, $g'(x)\neq 0, x\in[0,a]$ and $w$ is a real parameter whose absolute value can be extremely large.
%For the sake of simplicity, we will assume throughout this paper that $a>0$ and $w>0$.
If the integral is over another bounded domain $[a, b]$ with finitely many logarithmically singular points, it can
be written in sums of integrals of the form $I_{\log,w}^{[0,a]}[f,g]$ and non-singular integrals of the form $I_{w}^{[0,a]}[f,g]:=\int_0^af(x)e^{iwg(x)}dx$.
{\color{black}
Oscillatory integrals with some logarithmic singularities occurs frequently in the numerical
process of solving many problems of science and engineering such as electromagnetic and acoustic
scattering.
Since the antiderivatives of the integrands are unknown in most of cases, they have to
be computed numerically.
However, the high oscillation and the weak singularity of the integrands make
the classic numerical integral methods such as Gauss quadrature hard to derive an acceptable approximation
within a limited cost. The computation
of integrals of this type $I_{\log,w}^{[0,a]}[f,g]$ is regarded as a challenging issue which requires special focus.
}

Many effective methods have been proposed for the oscillatory integrals in order to overcome the difficulty caused by the high oscillation such as Filon-type methods \cite{ISERLES2004,ISERLES2005N,XIANG2007}, Levin methods \cite{LEVIN1982,OLVER2006}, the generalized quadrature rule \cite{EVANS2000}, numerical steepest descent methods \cite{HUYBRECHES2006}. We refer the interested reader to \cite{HUYBRECHES2009} for a review of these methods. {\color{black} There is a series of papers that develop quadratures for non-oscillatory integrals with singularities by using graded meshes \cite{1994XU}, or by Euler-Maclaurin summation formula \cite{ALPERT1995,KAPUR1997,ROKHLIN1990,STARR1991}.
An interesting hybrid Gauss-Trapezoidal quadrature rule  was introduced in  Alpert \cite{alpert1999hybrid} for the integrand with algebraic or logarithmic singularity, and for improper integrals with oscillatory weight $e^{i\gamma x
}$, where the quadrature nodes and weights  are computed by solving a nonlinear system. It is high time-consuming especially for highly oscillatory integral since the minimum sampling was  taken to be two points per period.  The hybrid Gauss-trapezoidal rule is quite accurate for integrand without highly oscillation, but fails to computation of the highly oscillatory integral when the frequency is much bigger than the number of the nodes (see Table 1).}

\begin{table}[!h]
 \tabcolsep 0pt \caption{Relative errors for the singular case $\int_0^1\left[\cos(wx)s(x)+\cos(wx+0.3)\right]dx$ with $s(x)=\log x$. Here the error is of order $O(h^{\ell}\log h)$ and $m=n+j+k$, where  $h=1/(n+a+b-1)$ and $(\ell,j,k,a,b)$ is shown and $f$ denotes the oversampling factor } \vspace*{-12pt}
 \textcolor{black} {
 \begin{center}{
 \footnotesize
\def\temptablewidth{1\textwidth}  {\rule{\temptablewidth}{1pt}}
\begin{tabular*}{\temptablewidth}{@{\extracolsep{\fill}} cccccc}
&&&$w=200$&&  \\   \hline
$m$ & $f$       &   (2,1,1,1,1)             & (4,3,2,2,2)            &   (8,7,4,5,4)         &  (16,16,8,10,7)        \\   \hline
70&1.10&  1.3983e-01       &    2.1661e-02        &  3.5405e-02     &1.3333e-04    \\
160&2.51&  1.5578e-02     &   1.8820e-03         &  2.7201e-05      &  2.0157e-10 \\
260&4.08&  4.1727e-03     &   2.5842e-04          &  7.7766e-07     &  3.6675e-15  \\ \hline
&&&$w=2000$&&  \\   \hline
$m$ & $f$       &   (2,1,1,1,1)             & (4,3,2,2,2)            &   (8,7,4,5,4)         &  (16,16,8,10,7)        \\   \hline
70&0.11& 54.5467      &    7.4588        &   135.0606     &144.2557    \\
160&0.25&   978.7887    & 69.4319      & 38.0776      &  40.4790 \\
260&0.41&  13.4481     & 36.3667          &  38.5161     &  4.1813  \\
  \end{tabular*}
{\rule{\temptablewidth}{1pt}}}
\end{center}
}
\end{table}

We next review the development of quadratures regarding both the properties of oscillation and logarithmic singularity.
The asymptotic behavior of Fourier integrals involving logarithmic singularities was obtained by repeated integration by parts in \cite{Erdelyi1956,MCKENNA1967}. This method is unstable.
To overcome the difficulty of the singularity  completely, a new Filon method was proposed in \cite{PIESSENS1992,XAINGGUO2014,DOMINGUEZ2014,KANG2015}
% \cite{KANG2013},
for the case when the modified moments including the singularity can be obtained numerically for the linear oscillator $g$.
 However, it might be impossible for the general cases of the oscillator $g$ since the modified moments are hard to be obtained.

 Recently,  by partition of the integration interval based on the singularity and the oscillation of the integrand, the composite Filon methods were considered in \cite{DOMINGUEZ2013} and \cite{MAXU2017} based on the Filon--Clenshaw--Curtis quadrature and moment-free Filon method \cite{XIANG2007}, respectively.
 {\color{black}
The  composite moment-free Filon-type methods developed in \cite{DOMINGUEZ2013,MAXU2017}  are  efficient for computing  oscillatory integrals with weakly singular integrand and stationary points. Furthermore, the new methods in  \cite{MAXU2017}, unlike the existing Filon-type methods,  do not have to compute the inverse of
the oscillator, and  have a polynomial order or exponential order of convergence. However,
 the main disadvantage of these composite methods is that subintervals near the singular point in the specified mesh have very small lengths and thus may cause serious round-off error problems.}

%One of the developed methods in the calculation of singular and oscillatory integrals is
% the Filon method, which is based on the modified moments which are the integrals of the product of polynomials, oscillatory functions and singular functions.
%
%Due to numerically stable recurrence relations for the modified moments, the proposed Filon method in \cite{KANG2013,KANG2015} was employed to compute oscillatory integrals with algebraic or logarithmic singularities at
%the end or interior points of the interval of integration.
%Specially, Filon-Clenshaw-Curtis methods have also been studied extensively \cite{} for singular and oscillatory integrals.
%Composite quadrature rules were developed recently in \cite{MAXU2017} to achieve the convergence of polynomial and exponential orders based on partitioning the integration domain according to the wave number and the singularity of the integrand.
A special  Gauss-type quadrature, based on the numerical steepest method, has been proposed for the highly oscillatory integrals with algebraic singularities \cite{HE2014,XU2015,Xu2016} for  linear oscillators, not applied to general oscillators.
{\color{black}
There is still much work to compute  the oscillatory integrals with logarithmic singularities in efficiency and accuracy.}

{\color{black}
Different from the existing methods, the purpose of this paper is
to design an efficient quadrature rule based on the classic Levin method for the computation of integrals of the form $I_{\log,w}^{[0,a]}[f,g]$.
The method developed in this paper requires no graded meshes and the computation of modified moments
and can be easily extended to the case with a complicated oscillator $g$ which inherit from the merits of the classic Levin method.
}

The Levin method, proposed in \cite{LEVIN1982},  is very efficient for computation of integrals of the form $I_{w}^{[0,a]}[f,g]$ if $f$ is not singular and $g'(x)\not=0$ for $x\in [a,b]$, where to evaluate the integral is  transformed into a certain ODE problem.
In particular,  one of the solutions of the ODE system is non-oscillatory and can be solved by a collocation technique. Compared with the other methods for the oscillatory integrals, the Levin method can be applied to a more general oscillator without explicit computation of the moments. It was also found that the Levin method is equivalent to the Filon method when applied to oscillatory integrals with a linear oscillator \cite{XIANG2007b}.
In addition, the Levin method can be implemented stably through the Chebyshev collocation method with TSVD \cite{LI2008} or by the GMRES method \cite{OLVER2010b,OLVER2010}. However, it cannot be applied directly to oscillatory integrals of the type $I_{\log,w}^{[0,a]}[f,g]$ with logarithmic singularity.
For example, we compute by the classic Levin method a simple logarithmically singular integral $\int_{0}^1 \log x e^{iwx}dx$
whose exact value is $-\frac{Si(w)}{w}-i\frac{\gamma-Ci(w)+\log w}{w}$ where $Si$ and $Ci$ denote the sine and cosine integral functions, respectively. Due to the singularity of the integrand at $x=0$, we use the modified Chebyshev-Gauss-Radau points $t_j=\left(1+\cos\frac{2\pi j}{2n-1}\right)/2, j=0,1,\ldots,n-1$ as the collocation points in the classic Levin method. The relative errors, shown in Table \ref{Ta0}, reveal that the classic Levin method loses the spectral accuracy when $w$ is small and it fails at all when the frequency is large enough.
\begin{table}[htb]
\begin{center}
\small
\caption{ Relative errors for $\int_{0}^1 \log x e^{iwx}dx$ computed by the classical Levin method with $n$ Chebyshev-Gauss-Radau collocation points}
\label{Ta0}
\def\temptablewidth{1\textwidth}  {\rule{\temptablewidth}{1pt}}
%\begin{tabular}{ccccc}
\begin{tabular*}{\temptablewidth}{@{\extracolsep{\fill}} ccccc}
%\hline
$n$& $w=10$  & $w=10^2$  & $w=10^3$  & $w=10^4$ \\
\hline
4&$2.1704e-01$&$4.1370e-01$&$5.6065e-01$&$6.5452e-01$\\
8&$1.4121e-03$&$2.0756e-01$&$3.8046e-01$&$5.0749e-01$\\
16&$7.0286e-04$&$7.0034e-02$&$2.2451e-01$&$3.7064e-01$\\
32&$1.7388e-04$&$1.0253e-02$&$1.0327e-01$&$2.4454e-01$\\
64&$4.1838e-05$&$2.5933e-04$&$2.6163e-02$&$1.3555e-01$\\
%\hline
\end{tabular*}
 {\rule{\temptablewidth}{1pt}}
\end{center}
\end{table}

{\color{black}
Our main idea of the new Levin method is the separation of the singularity from the solution of the singular ODE to avoid the influence from the singular forcing function.
%, here, we consider  a new Levin method to $I_{\log,w}^{[0,a]}[f,g]$ by separation of the singularity
By using the technique of singularity separation, the solution of the singular ODE is transformed into solutions of two non-singular ODEs based on the principle of superposition. The linear systems obtained in solving these two ODEs by the collocation method share the same matrix and thus they can be solved efficiently with a little increased cost compared with the classic Levin method. The new Levin method for oscillatory integrals with logarithmic singularities keeps nearly all of the merits of the classic Levin method for the non-singular oscillatory integrals:
\begin{enumerate}
  \item[(1)] It does not require the computation of moments and is applicable for the nonlinear oscillator case.
  \item[(2)] It converges nearly superalgebraically with respect to the number of collocation points when $f$ is smooth, $g$ is linear and the ODEs are solved by the Chebyshev collocation method. Numerical experiments validate that it is also true for the general oscillator $g$ without stationary points.
  \item[(3)] It shares the important property: the higher of the frequency the more accuracy of the method. The asymptotic order is $\bO(w^{-2}\log(1+|w|))$ with respect to the frequency.
\end{enumerate}
}

This paper is organized as follows. In section 2, we develop the new Levin method for integrals of the type $I_{\log,w}^{[0,a]}[f,g]$ and present two algorithms for linear and nonlinear oscillators, respectively. The error analysis is presented in section 3 for each algorithm of the new Levin method. Numerical examples are shown in section 4 to validate the proposed Levin method. We summarize our conclusions in section 5.

\section{A new Levin method for logarithmically singular oscillatory integrals}
In this section, we propose a new Levin method for logarithmically singular oscillatory integrals with two algorithms for linear and nonlinear oscillators, respectively.
{\color{black}
We further assume that the oscillator satisfies $g(0)=0$ and $g'(x)>0$, $x\in [0,a]$, i.e. $g$ is an increasing function starting from the origin. If $g(0)\neq 0$, then $g(x)$ is replaced by $g(x)-g(0)$ and if $g'(x)<0$, then $g(x)$ is replaced by $-g(x)$ and $w$ by $-w$. }
For the notation simplification, $I_{\log,w}^{[0,a]}[f,g]$ is shorten as $I_{\log}[f,g]$ in this section.

%Let $I_\omega^{[a,b]}[f,g]=\int_a^b f(x)e^{iwg(x)}dx$ where $w$ is a parameter with large absolute value, $f$ and $g$ are suitably smooth functions.

The spirit of the Levin method for
$$
I_\omega^{[a,b]}[f,g]=\int_a^b f(x)e^{iwg(x)}dx
$$
is based upon the fact that if $f$ were of the form
\[
f(x)=p'(x)+iwg'(x)p(x)\equiv \bL p(x),\; a\leq x\leq b,
\]
then the integral $I_\omega^{[a,b]}[f,g]$ could be evaluated as
\[
I_\omega^{[a,b]}[f,g]=\int_a^b (p'(x)+iwg'(x)p(x))e^{iwg(x)}dx=p(b)e^{iwg(b)}-p(a)e^{iwg(a)}.
\]
Thus it transforms the evaluation of the integral into an ODE problem, $\bL p(x)=f(x)$.
It has been proven in \cite{LEVIN1982,LEVIN1997} that the ODE possesses at least a non-oscillatory solution
which inspires
the Levin method, or the Levin collocation method.
That is to find a polynomial $p_n(x)$ with degree $\leq n-1$ such that
\begin{equation}\label{levinc}
\bL p_n(c_j)=f(c_j),\; j=0,1,\ldots,n-1,
\end{equation}
and to compute $I_\omega^{[a,b]}[f,g]$ numerically by
\begin{equation}\label{levin}
\int_a^b \bL p_n(x)e^{iwg(x)}dx =p_n(b)e^{iwg(b)}-p_n(a)e^{iwg(a)}.
\end{equation}

The convergence rate of the Levin method has been studied extensively \cite{LEVIN1997,XIANG2007b,OLVER2010}. In addition, the Filon method and the Levin method were proved to be identical when the oscillator is a linear function in \cite{XIANG2007b}. Furthermore, two numerically stable algorithms,  the Levin-Chebyshev collocation method using TSVD and the GMRES-Levin collocation method are presented in  \cite{LI2008}  and \cite{OLVER2010}, respectively.

To compute $I_{\log}[f,g]$ according to the spirit of the Levin method, we hope to find a function $p$ such that $\left(p(x)e^{iwg(x)}\right)'=f(x)\log(x)e^{iwg(x)}$, or
\begin{equation}\label{lode}
\bL p(x)=f(x)\log(x).
\end{equation}
Then $I_{\log}[f,g]=p(a)e^{iwg(a)}-p(0)e^{iwg(0)}$. However,
the direct application of the Levin-Chebyshev collocation method
 to the computation of $I_{\log}[f,g]$ shall cause the great error.
It is because that the logarithmically singular function $\log(x)$ leads to a singular Levin ODE which can not be solved efficiently by
the collocation method based on polynomials.
The main purpose of this paper is to address this difficulty and to propose efficient new Levin methods for singular integrals of this kind.

The basic idea in solving the singular ODE of the Levin method is to separate the singularity and to transform the singular ODE to non-singular ODEs. Inspired by the equality $\left(x(\log x-1)\right)'=\log x$, the solution $p$ should also possess the logarithmic singularity.
With this property, we introduce a new form for the function $p$
\begin{equation}\label{form}
p(x)=q(x)\log(x)+h(x).
\end{equation}
By substituting \eqref{form} into \eqref{lode}, we obtain that
\begin{equation}\label{eq1}
    \left(q'(x)+iwg'(x)q(x)-f(x)\right)\log(x)+h'(x)+iwg'(x)h(x)+\frac{q(x)}{x}=0
\end{equation}
To obtain a particular solution $p$, we consider naturally a particular case of \eqref{eq1},
\begin{eqnarray}
  q'(x) + iwg'(x) q(x) &=& f(x),\label{eq3}\\
  h'(x)+iwg'(x)h(x) &=& -\frac{q(x)}{x}. \label{eq4}
\end{eqnarray}
Once we get the particular solutions $q$ and $h$ of \eqref{eq3} and \eqref{eq4}, respectively, we have a particular solution for  \eqref{lode}.
%By introducing the representation of $p$, we hope to separate the singularity of the ODE and to transform the singular one to a non-singular one.

We next consider the system of ODEs \eqref{eq3} and \eqref{eq4}. {\color{black} For this purpose, we recall a result in \cite{LEVIN1997} about the existence of a `non-oscillatory' solution of the ODE \eqref{eq3} which is understood as a solution whose many derivatives are bounded uniformly for $w$.
\begin{lemma}\label{Existence}
Let $f\in C^{2n+1}[0,a]$, $g\in C^{2n+1}[0,a]$ satisfies $g'(x)\neq 0$ $x\in[0,a]$ and the $(2n+1)$ derivatives of $g$ are bounded uniformly in $w$, for $|w|>1$. Then there exists a constant $c$ independent of $w$ and a solution $q(x)$ of  \eqref{eq3} satisfying
\begin{equation}\label{condq}
  \left\|\bD^jq\right\|<\frac{c}{w}, j=0,\ldots,n,
\end{equation}
where $\bD$ is the differential operator and $\|\cdot\|$ denotes the maximum norm on $[0,a]$.
\end{lemma}

It is clear from Lemma \ref{Existence} that there exists a particular non-oscillatory solution satisfying the ODE \eqref{eq3} which can be approximated well by polynomials based on the collocation method no matter how large is the absolute value of $w$.}
 However, if the solution obtained from \eqref{eq3} does not vanish at $x=0$, equation \eqref{eq4} is still a strongly singular ODE. To {\color{black} loose this singularity of \eqref{eq4},} we restrict $q(0)=0$.
{\color{black} It is proved in the following lemma that $\frac{q(x)}{x}$ possesses good regularity when $q$ is smooth enough with the restriction $q(0)=0$. Note that when $x=0$, the value of $\frac{q(x)}{x}$ is defined by taking the limitation as $x\rightarrow 0$.
 \begin{lemma}\label{lemma1}
   If $q\in C^n[0,a]$ and $q(0)=0$, $1\leq n\in\aN$, then  $\tilde{q}\in C^{n-1}[0,a]$ where
   \begin{equation*}
     \tilde{q}(x):=\begin{cases}\frac{q(x)}{x}, x\neq 0,\\ q'(0), x=0. \end{cases}
   \end{equation*}
 \end{lemma}
 \begin{proof}
   When $x\neq 0$, it is deduced directly from  the general Leibniz rule for the high derivative of a product of two factors that
   \begin{equation*}
   \tilde{q}^{(k)}(x)=\frac{1}{x^{k+1}}\sum_{j=0}^kC_k^j(-1)^jj!x^{k-j}q^{(k-j)}(x), 0\leq k<n\;\text{and}\;k\in\aN.
   \end{equation*}
   We next verify by induction on $k$ that when $x=0$, the $k$-th derivative of $\tilde{q}$ equals $\frac{q^{(k+1)}(0)}{k+1}$ for $0\leq k<n$. When $k=0$, it is obvious according to the definition of $\tilde{q}$. Assuming that the case $k-1$ is established for $1\leq k<n$, i.e. $\tilde{q}^{(k-1)}(0)=\frac{q^{(k)}(0)}{k}$, we then consider the $k$-derivative of $\tilde{q}$ at $x=0$. According to the definition of the $k$-derivative, it is known that
   \begin{equation*}
   \begin{split}
     \tilde{q}^{(k)}(0) & =\lim_{x\rightarrow 0}\frac{\tilde{q}^{(k-1)}(x)-\tilde{q}^{(k-1)}(0)}{x} \\
     & = \lim_{x\rightarrow 0}\frac{1}{x^{k+1}}\left( \sum_{j=0}^{k-1} C_{k-1}^j(-1)^jj!x^{k-1-j}q^{(k-1-j)}(x)-\frac{q^{(k)}(0)}{k}x^k\right) \\
     &=\lim_{x\rightarrow 0}\frac{q^{(k)}(x)-q^{(k)}(0)}{(k+1)x}=\frac{q^{(k+1)}(0)}{k+1},
   \end{split}
   \end{equation*}
   where the L'Hospital's rule was used. Again by the use of the L'Hospital's rule, it is easily validated that $\lim_{x\rightarrow 0}\tilde{q}^{(k)}(x)=\tilde{q}^{(k)}(0)$ for $k<n$. Thus $\tilde{q}\in C^k[0,a]$ for $k<n$ and the desired result follows by setting $k=n-1$.
 \end{proof}
 }

We thus obtain an initial problem
\begin{equation}\label{ode}
    q'(x)+iwg'(x)q(x)=f(x),\; x\in[a,b],\;\;\text{and}\; q(0)=0,
\end{equation}
whose exact solution is represented by the highly oscillatory integral, i.e.
\[
q(x)=\int_0^xf(t)e^{iw(g(t)-g(x))}dt.
\]
Instead of computing the integrals directly, {\color{black} we find a non-oscillatory and well-behaved particular solution $q_1$ satisfying the ODE \eqref{eq3} without the initial condition according to Lemma \ref{Existence} } and the expression for $q$, the solution of the initial problem \eqref{ode}, is then given by
\begin{equation}\label{odesol}
 q(x)=q_1(x)-q_1(0)e^{-iwg(x)}.
 \end{equation}

We turn to equation \eqref{eq4}. {\color{black} Note that when $f$ is sufficiently smooth, equation \eqref{eq4} is non-singular. However, the solution $q$ obtained from \eqref{ode} is still highly oscillatory which would influence the solution of \eqref{eq4}. To deal with the high oscillation separately,} the solution $q$ is rewritten in the form
\begin{equation}\label{structure}
    q(x)=q_2(x)x+q_1(0)(1-e^{-iwg(x)}),
\end{equation}
where $q_2$ is determined by
\begin{equation}\label{q2}
q_2(x)=\frac{q_1(x)-q_1(0)}{x}, x\neq 0,\;\; \text{and} \; q_2(0)=q_1'(0)=f(0)-iwg'(0)q_1(0).
\end{equation}
{\color{black} It is known from Lemma \ref{lemma1} and its proof that $q_2$ is non-oscillatory and has good regularity if $q_1$ is non-oscillatory and behaves well. Especially, when $q_1$ is a polynomial, so is $q_2$. }

With the substitution of $q$ in \eqref{eq4}, we get that
\begin{equation}\label{eq7}
    h'(x)+iwg'(x)h(x) = -q_2(x)-q_1(0)\frac{1-e^{-iwg(x)}}{x}.
\end{equation}
By the linear superposition, the solution $h(x)$ can be split into two parts $h(x)=h_1(x)+h_2(x)$ where $h_j, j=1,2$ satisfy, respectively, the equations
\begin{eqnarray}
  h_1'(x)+iwg'(x)h_1(x) &=& -q_2(x), \label{eq5}\\
  h_2'(x)+iwg'(x)h_2(x) &=& -q_1(0)\frac{1-e^{-iwg(x)}}{x}. \label{eq6}
\end{eqnarray}
According to Lemma \ref{lemma1} and the property of $q_1$, the ODE \eqref{eq5} is a non-singular equation which can be solved efficiently by the classic Chebyshev collocation method with the TSVD.

We next focus on equation \eqref{eq6} to present an analytic particular solution for \eqref{eq6} when $g(x)=x$. Since possessing an oscillatory forcing function, the equation \eqref{eq6} is very difficult to be solved numerically for general oscillator $g$ due to the oscillatory properties from $iwg'$ and the oscillatory forcing function. To get around the obstacle, we turn to the special functions for help when $g$ is linear and consider the case of general oscillators later. To this end, we introduce the complementary incomplete Gamma function and its property. The complementary incomplete Gamma function, denoted as $\Gamma(\alpha,z)$, is defined by
\[
\Gamma(\alpha,z) = \int_z^\infty e^{-t} t^{\alpha-1} dt,
\]
and it has a known series expansion when $\alpha=0$,
\begin{equation}\label{expansion}
\Gamma(0,z)=-\gamma-\Log(z)-\sum_{j=1}^\infty \frac{(-z)^j}{j(j!)},
\end{equation}
where $\gamma$ is the Euler's constant, equaling approximately 0.57721566490153286060651 and {\color{black} $\Log(z):=\log(|z|)+i\arg(z)$ and $\arg(z)$ denotes the principle argument of $z$ for a complex $z$.}. We are now ready to derive a particular solution of \eqref{eq6} for the case with $g(x)=x$. It is well-known that the solution of \eqref{eq6} satisfying $h_2(0)=0$ has a closed form,
 \[
h_2(x)=q_1(0) e^{-iwx} \int_0^x \frac{1-e^{iwt}}{t}dt.
\]
Substituting the Taylor's series expansion of $e^{iwt}$ and using the expansion \eqref{expansion} of $\Gamma(0,z)$, it is obtained that
\[
\begin{split}
\int_0^x \frac{1-e^{iwt}}{t}dt=-\sum_{j=1}^\infty \frac{(iwx)^j}{j(j!)} =\gamma+\Gamma(0,-iwx)+\Log(-iwx).
\end{split}
\]

Thus a particular solution  for \eqref{eq6} when $g(x)=x$ is given by
\begin{equation}\label{h2}
h_2(x)=q_1(0)e^{-iwx}\left( \gamma+\Gamma(0,-iwx)+\Log(-iwx) \right),x\neq 0,\;\;\text{and} \; h_2(0)=0.
\end{equation}

 Combining the solutions of \eqref{eq3}, \eqref{eq5} and \eqref{eq6}, a particular solution $p$ is derived successfully which reads
\begin{equation}\label{p}
p(x)=\left(q_1(x)-q_1(0)e^{-iwx}\right)\log x+h_1(x)+h_2(x).
\end{equation}
The logarithmically singular and oscillatory integral with linear oscillator follows directly
\begin{equation}\label{inte}
\int_0^a f(x) \log(x) e^{iwx} dx=p(a)e^{iwa}-p(0).
\end{equation}
Specially, when $a=1$, $\int_0^a f(x) \log(x) e^{iwx} dx=(h_1(1)+h_2(1))e^{iw}-h_1(0)$.

We summarize the first algorithm of the new Levin method for logarithmically oscillatory integrals with $g(x)=x$ as follows.
\begin{algorithm}[Levin algorithm for a linear oscillator]\label{alg}
Given a function $f\in C^1[0,a]$, $g(x)=x$ and a positive integer $n$, where  $\Bx_T $ denotes the vector of the Chebyshev-Lobatto points, i.e. the $j$-th element of $\Bx_T $, $x_j:=-\cos\frac{j\pi}{n-1}, j=0,1,\ldots,n-1$:

    {\bf 1:} Obtain $\hat{\Bx}=\phi(\Bx_T)$ where $\phi(x)=\frac{a}{2}x+\frac{a}{2}$;

    {\bf 2:} Let $\Bf=f(\hat{\Bx})$ and $G=\diag(g'(\hat{\Bx}))$;

    {\bf 3:} Construct the matrix $L=\frac{2}{a}D+i\omega G$;

    {\bf 4:} Solve $\Bq_1=L^{-1}\Bf$ by TSVD;

    {\bf 5:} Construct the vector $\Bq_2=q_2(\hat{\Bx})$ where $q_2$ is defined in \eqref{q2};
    %X(\Be(\phi^{-1}(\hat{\Bx}))^\top \Bq_1 -\Be(-1)^\top \Bq_1)$ with the first element replaced by $\Be(-1)^\top (\Bf-iw\Bq_1)$ where the matrix $X=\diag(1/\hat{\Bx})$;

    {\bf 6:} Solve $\Bh_1=L^{-1}\Bq_2$ by TSVD;

    {\bf 7:} Derive the value $h_2(a)$ by the formula \eqref{h2};% with $q_1(0)$ replaced by $\Be(-1)^\top \Bq_1$;

    {\bf 8:} Define
{\color{black}    \begin{equation}\label{levinwang}
    Q_{\log,w,n}^{[0,a],L}[f]=\left(e^{iwa}\Be_n^\top-\Be_1^\top\right)(\Bq_1\log a+\Bh_1)+ e^{iwa} h_2(a),
    \end{equation}
where $\Be_j$ denotes a unit column vector of size $n\times1$ whose $j$-th element is 1 while the others 0.}
    %{\bf 1:} Solving \eqref{eq3} by collocation method to obtain a non-oscillatory function $q_1$,
%
%    {\bf 2:} Using formula \eqref{q2} to obtain function $q_2$,
%
%    {\bf 3:} Solving \eqref{eq5} by collocation method to obtain a non-oscillatory function $h_1(x)$,
%
%    {\bf 4:} Combing equation \eqref{h2} and \eqref{p} to obtain a numerical approximation for integral $\int_0^af(x)\log(x)e^{iwx^n}dx$.
\end{algorithm}

We next consider the second algorithm of the new Levin quadrature for the case with a general oscillator $g(x)$. Since it is hard to derive $h_2$ from \eqref{eq6} numerically or analytically for the general oscillator, the Levin method for the linear case can not be applied directly.
To overcome this difficulty, we split the integral $I_{\log}[f,g]$ into two parts,
     \begin{equation}\label{nonlinear}
    \begin{split}
        \int_0^a f(x)\log(x)e^{iwg(x)}dx=& \int_0^af(x)\log\frac{x}{g(x)}e^{iwg(x)}dx +\int_0^af(x)\log(g(x)) e^{iwg(x)}dx \\
        \triangleq &I_1(f)+I_2(f).
        \end{split}
    \end{equation}
By the Hospital's rule, there exists the limit
   \[
   \lim_{x\rightarrow 0}\frac{x}{g(x)}=\frac{1}{g'(0)}\neq 0.
   \]
{\color{black} Besides, the product function $f(x)\log\frac{x}{g(x)}$ has good regularity if $f$ and $g$ is suitably smooth according to Lemma \ref{lemma1} since $\log\frac{x}{g(x)}=-\log\frac{g(x)}{x}$.}  It reveals that the integral $I_1(f)$ is readily computed efficiently by the classic Levin method.

For the second integral $I_2(f)$, we follow the same idea of singularity separation and pursue a particular solution $p$ with the form $p(x)=q(x)\log(g(x))+h(x)$ such that
\begin{equation}
    p'(x)+iwg'(x)p(x)=f(x)\log(g(x)).
\end{equation}
Similarly, it is
obtained two ODEs for $q(x)$ and $h(x)$,
\begin{eqnarray}
  q'(x) + iwg'(x) q(x) &=& f(x), \; q(0)=0,\label{eq16}\\
  h'(x)+iwg'(x)h(x) &=& -\frac{q(x)g'(x)}{g(x)}. \label{eq17}
\end{eqnarray}
The solution $q$ can be represented by a sum of a non-oscillatory particular solution and a multiple of the general solution. To facilitate the solution of equation \eqref{eq17}, the solution $q(x)$ is formed as follows
 $$q(x)=q_1(x)-q_1(0)e^{-iwg(x)}=q_2(x)g(x)+q_1(0)(1-e^{-iwg(x)})$$
 where $q_1$ is a function satisfying the ODE
\begin{equation}\label{eq24}
q_1'(x)+iw g'(x)q_1(x) = f(x)
\end{equation}
and
\begin{equation}\label{q22}
q_2(x)=\frac{q_1(x)-q_1(0)}{g(x)}, x\neq 0,\;\; \text{and} \; q_2(0)=q_1'(0)= \frac{f(0)-iwg'(0)q_1(0)}{g'(0)}.
\end{equation}
Lemmas \ref{Existence} and \ref{lemma1} indicate that $q_1$ and $q_2$ are non-oscillatory and possess good regularity when $f$ and $g$ is smooth enough.

By substituting the expression of $q$ in \eqref{eq17}, the ODE \eqref{eq17} is broke into two ODEs,
\begin{eqnarray}
  h_1'(x)+iwg'(x)h_1(x) &=& -q_2(x) g'(x) \label{eq18}\\
  h_2'(x)+iwg'(x)h_2(x) &=& -q_1(0)g'(x)\frac{1-e^{-iwg(x)}}{g(x)} \label{eq19}
\end{eqnarray}
and then a solution of $h$ is the sum of $h_1$ and $h_2$.
 With the help of the complementary incomplete Gamma function, a particular solution of ODE \eqref{eq19} is given explicitly by
     \begin{equation}\label{eq23}
     h_2(x)=q_1(0)e^{-iwg(x)}\left( \gamma+\Gamma(0,-iwg(x))+\Log(-iwg(x)) \right),x\neq 0,
     \end{equation}
 and $h_2(0)=0$.
 Since ODEs \eqref{eq24} and \eqref{eq18} are non-singular and possess at least one non-oscillatory solution, they can be solved efficiently by the Chebyshev-collocation methods. Once $q_1$, $h_1$ and $h_2$ are obtained, the integral $I_2(f)$ is readily computed according to the spirit of Levin idea.

We conclude the algorithm of the new Levin method for integrals with a general oscillator $g$. Let $f_1(x)= f(x)\log\frac{x}{g(x)}$ for $x\neq 0$ and $f_1(0)=f(0)\log\frac{1}{g'(0)}$.
\begin{algorithm}[Levin algorithm for a general oscillator] \label{alg2}
Given a function $f\in C^1[0,a]$, $g(x)\in C^1[0,a]$, $g'(x)> 0$ and $g(0)=0$ and  a positive integer $n$:

    {\bf 1:} Obtain $\hat{\Bx}=\phi(\Bx_T)$ where $\phi(x)=\frac{a}{2}x+\frac{a}{2}$;

    {\bf 2:} Let $\Bf_1=f_1(\hat{\Bx})$, $\Bf=f(\hat{\Bx})$ and $G=\diag(g'(\hat{\Bx}))$;

    {\bf 3:} Construct the matrix $L=\frac{2}{a}D+i\omega G$;% and $L=\frac{2}{a}D+i\omega I_n$ where $I_n$ is the identity matrix of $n$ diemnsions;

    {\bf 4:} Solve $\Bq=L^{-1}\Bf_1$ and $\Bq_1=L^{-1}\Bf$ by TSVD;

    {\bf 5:} Construct the vector $\Bq_2=-Gq_2(\hat{\Bx})$ where $q_2$ is defined in \eqref{q22};
    %=X(\Be(\phi^{-1}(\hat{\Bx}))^\top \Bq_1 -\Be(-1)^\top \Bq_1)$ with the first element replaced by $\Be(-1)^\top (\Bf-iw\Bq_1)$ where the matrix $X=\diag(1/\hat{\Bx})$;

    {\bf 6:} Solve $\Bh_1=L^{-1}\Bq_2$ by TSVD;

    {\bf 7:} Derive the value $h_2(a)$ by the formula \eqref{eq23};% with $q_1(0)$ replaced by $\Be(-1)^\top \Bq_1$;

    {\bf 8:} Define
    {\color{black}    \begin{equation}\label{levinwang2}
    \begin{split}
    Q_{\log,w,n}^{[0,a],L}[f,g]=&\left(e^{iwg(a)}\Be_n^\top - \Be_1^\top\right) \left(\Bq +\Bq_1\log (g(a))+\Bh_1\right)+e^{iwg(a)}h_2(a),
    \end{split}
    \end{equation}
    where $\Be_j$ denotes a unit column vector of size $n\times1$ whose $j$-th element is 1 while the others 0.}
\end{algorithm}

{\color{black} In a word, the new Levin method adopts the separation of singularity and oscillation to get around the singular difficulty by transforming the singular ODE into three non-singular ODEs with two of whom can be solved efficiently by the Chebyshev-collocation methods and the other one is solved analytically. It makes the Levin idea applicable for oscillatory integrals with logarithmic singularities. }

\section{Error analysis}
In this section, we present error analysis for the two algorithms of the new Levin method proposed in section 2.

We first reveal the relationship between the new Levin method and the Filon method when computing oscillatory integrals with a linear oscillator $g(x)=x$.
To this end, denote the numerical solutions of $q_1$, $q_2$ in \eqref{structure} and $h_1$ in \eqref{eq5} obtained through Algorithm \ref{alg} based on arbitrary points $\hat{\emph{\Bx}}=\{0\leq\hat{x}_0<\hat{x}_1<\ldots<\hat{x}_{n-1}\leq a\}$ by $\hat{q}_1$, $\hat{q}_2$ and $\hat{h}_1$. It is obvious that
 $$\hat{q}(x)=\hat{q}_1(x)-\hat{q}_1(0) e^{-iwx}=\hat{q}_2(x)x+\hat{q}_1(0)(1-e^{-iwx}).$$
Let $\hat{h}_2$ denote the approximation of $h_2$ in \eqref{h2} with $q_1(0)$ replaced by $\hat{q}_1(0)$.
The algorithm of the new Levin method for $\int_0^a f(x)\log(x)e^{iwx}dx$ can be expressed as
\begin{equation}\label{nlevin}
\begin{split}
    Q_{\log,w,n}^{[0,a],L}[f]&=\hat{q}(x) e^{iwx}\log(x)|^a_0+\hat{h}_1(x)e^{iwx}|^a_0+\hat{h}_2(x)e^{iwx}|^a_0\\
    & =\hat{q}(x) e^{iwx}\log(x)|^a_0+\int_0^a\bL(\hat{h}_1+\hat{h}_2)e^{iwx}dx
    \end{split}
\end{equation}
Let $\hat{f}$ denote the interpolant of $f$ of degree $n-1$ interpolating on the same points $\hat{\emph{\Bx}}$. The Filon method for $\int_0^a f(x)\log(x)e^{iwx}dx$ is to calculate
\begin{equation}\label{filon}
    Q_{\log,w,n}^{[0,a],F}[f]:=\int_0^a\hat{f}(x)\log(x)e^{iwx}dx.
\end{equation}

In the following, we present the relation between $Q_{\log,w,n}^{[0,a],L}[f]$ and $Q_{\log,w,n}^{[0,a],F}[f]$.
{\color{black}
\begin{theorem}\label{feql}
The new Levin method and the Filon method are identical in the computation of $\int_0^a f(x)\log(x) e^{iwx}dx$ when they are based on the same interpolation points $\hat{\Bx}$.
\end{theorem}
}
\begin{proof} When $\hat{q}_1$ is obtained by polynomial interpolation, it is a polynomial of degree $n-1$ and thus
$\hat{q}_2$ is a polynomial of degree less than $n-1$. By the Fundamental Theorem of Algebra, there exists
\begin{equation}\label{eq10}
\bL \hat{h}_1=-\hat{q}_2.
\end{equation}
With the definition \eqref{h2} of $\hat{h}_2$, we have that
\begin{equation}\label{eq11}
    \bL \hat{h}_2(x)= -\hat{q}_1(0)\frac{1-e^{-iwx}}{x}.
\end{equation}
 Using the relation between $\hat{q}$ and $\hat{q}_2$, it is derived directly from \eqref{eq10} and \eqref{eq11} that
\begin{equation}\label{eq12}
    \bL (\hat{h}_1+\hat{h}_2)(x)=-\frac{\hat{q}(x)}{x}.
\end{equation}
Combining \eqref{nlevin} and \eqref{eq12}, it is obtained that
\begin{equation}\label{eq13}
    Q_{\log,w,n}^{[0,a],L}[f]=\hat{q}(x) e^{iwg(x)}\log(x)|^a_0-\int_0^a \frac{\hat{q}(x)}{x}e^{iwx}dx=\int_0^a \bL \hat{q}(x)\log (x)e^{iwx}dx.
\end{equation}
where the second equality is assured by integration by parts.

Using the Fundamental Theorem of Algebra again, it has been proven in \cite{XIANG2007b} that
\begin{equation}\label{eq8}
\bL \hat{q}_1=\hat{f}.
\end{equation}
Since $\bL (\hat{q}_1(0)e^{-iw\cdot})=0$, it is obtained from \eqref{eq8} and the equality $\hat{q}(x)=\hat{q}_1(x)-\hat{q}_1(0) e^{-iwx}$  that
\begin{equation}\label{eq9}
\bL \hat{q}=\hat{f}.
\end{equation}

Equations \eqref{eq9} and \eqref{eq13} finally confirm the equivalence between the new Levin method and the Filon method, i.e. $Q_{\log,w,n}^{[0,a],L}[f]=Q_{\log,w,n}^{[0,a],F}[f]$.
\end{proof}

It has been proved in \cite{XIANG2007b} that the Levin method is equivalent to the Filon method when calculating non-singular oscillatory integrals with the linear oscillator.
Theorem \ref{feql} tells that this property is also kept for the new Levin method when dealing with the singular oscillatory integrals with the linear oscillator.
Note that Theorem \ref{feql} is only true for the case when $g(x)$ is linear which is needed in \eqref{eq10} and \eqref{eq8} in the proof. With this equivalence, the error analysis of the new Levin algorithms is readily obtained.

We next present the error analysis for Algorithm \ref{alg} with arbitrary points $\hat{\emph{\Bx}}$ when $g(x)=x$ based on Theorem \ref{feql}.
Let $E_n(f):=\left|I_{\log,w}^{[0,a]}[f,\ell]-Q_{\log,w,n}^{[0,a],L}[f]\right|$ denote the absolute error where $\ell(x):=x$. For the purpose of bounding $E_n(f)$, we recall two basic lemmas in the numerical analysis for the computation of oscillatory integrals.
{\color{black}
\begin{lemma}(van der Corput-type lemma \cite[p.332,334]{STEIN1993} \cite{XAINGGUO2014})\label{xlemma}
Suppose that $f\in C^1[0,a]$ and $g\in C^2[0,a]$ satisfying $|g'(x)|\geq 1, x\in[0,a]$ and $g'(x)$ is monotonic, then for all $w>0$, there exists a constant $C$ independent of $w$
such that
\begin{eqnarray*}
  \left|\int_0^a f(x) e^{iwg(x)}dx \right| &\leq & Cw^{-1} \left( |f(a)|+\int_0^a|f'(x)|dx \right), \\
  \left|\int_0^a  \ln (x) f(x) e^{iwx}dx \right| &\leq & C(1+|\ln(w)|) w^{-1} \left( |f(a)|+\int_0^a|f'(x)|dx \right).
\end{eqnarray*}
\end{lemma}

We are ready to analyze the absolute error $E_n(f)$. Let $C$ denote a generic
constant independent of $n$ and $w$ whose value may be changed in each appearance.
\begin{theorem}\label{thm1}
If $f\in C^{n}[0,a]$ and $n\geq 2$, then the numerical integral computed by Algorithm \ref{alg} based on arbitrary points $\hat{\Bx}=\{0\leq\hat{x}_0<\hat{x}_1<\ldots<\hat{x}_{n-1}\leq a\}$ satisfies
\begin{equation}\label{error}
    E_n(f)\leq C(1+|\ln(w)|) w^{-1} \frac{\|f^{(n)}\|_\infty a^n}{(n-1)!}.
\end{equation}
Specially, when $\hat{x}_0=0$, $\hat{x}_{n-1}=a$ and $n\geq 3$, i.e. both endpoints are included, there exists
\begin{equation}\label{error2}
    E_n(f)\leq C(1+|\ln(w)|) w^{-2} \frac{\|f^{(n)}\|_\infty a^{n-1}}{(n-2)!}.
\end{equation}
\end{theorem}
\begin{proof}
Theorem \ref{feql} reveals that
\begin{equation}\label{eq14}
    E_n(f)=\left|\int_0^a (f(x)-\hat{f}(x))\log(x)e^{iwx}dx\right|
\end{equation}
where $\hat{f}$ is the interpolation of $f$ on the nodes $0\leq \hat{x}_0<\hat{x}_1<\ldots<\hat{x}_{n-1}\leq a$. In order to estimate the error,
let $\psi(x):=f(x)-\hat{f}(x)$. It is obvious that $\psi(\hat{x}_j)=0, j=0,1,\ldots,n-1$. According to Rolle's theorem, there exist $y_j\in (\hat{x}_j, \hat{x}_{j+1})$ such that
\[
\psi'(y_j)=0,\; j=0,1,\ldots,n-2.
\]
Using the expression for interpolation errors, it is clear that
\[
\psi(x)=\frac{\psi^{(n)}(\xi_1)}{n!}\prod_{j=0}^{n-1} (x-\hat{x}_j), \;\; \psi'(x)=\frac{\psi^{(n)}(\xi_2)}{(n-1)!}\prod_{j=0}^{n-2} (x-y_j)
\]
where $\xi_1,\xi_2\in[0,a]$ depending on the value of $x$.
According to Lemma \ref{xlemma}, there exists a constant $C$ independent of $n$ and $w$ such that
%Combining functions $\psi, \psi'$, Lemma \ref{stein} for $k=1$ and $c(1)=3$, and the fact $\hat{f}^{(n)}\equiv 0$, we get
\begin{equation}\label{eq15}
\begin{split}
E_n(f)&=\left|\int_0^a \psi(x)\log(x)e^{iwg(x)}dx\right| \\
& \leq C(1+|\ln(w)|) w^{-1} \left( |\psi(a)|+\int_0^a|\psi'(x)|dx \right) \\
& \leq C(1+|\ln(w)|) w^{-1} (\|\psi\|_\infty+a\|\psi'\|_\infty).
\end{split}
\end{equation}
Since $\hat{f}^{(n)}\equiv 0$, the desired inequality \eqref{error} follows directly.

When $x_0=0$ and $x_{n-1}=a$, there have $\psi(0)=0$ and $\psi(a)=0$. We can derive by integration by parts that
\begin{equation*}
    \int_0^a\psi(x)\log(x) e^{iwx}dx=-\frac{1}{iw} \int_0^a e^{iwx} \left( \psi'(x)\log (x)+ \frac{\psi(x)}x \right) dx.
\end{equation*}
Again using  Lemma \ref{xlemma}, we obtain
\begin{equation*}
   \begin{split}
E_n(f)&\leq C(1+|\ln(w)|) w^{-2} \left(|\psi'(a)|+\int_0^a\left|\psi''(x)\right| dx +\left| \frac{\psi(a)}{a}\right|+ \int_0^a \left|\left(\frac{\psi(x)}{x}\right)'\right| dx\right)\\
&\leq C(1+|\ln(w)|) w^{-2} \left(\|\psi'\|_\infty+ a\|\psi''\|_\infty+\|\tilde{\psi}\|_\infty+a \|\tilde{\psi}'\|_\infty \right),
\end{split}
\end{equation*}
where $\tilde{\psi}(x):=\psi(x)/x$. By Taylor's expansions, there exists $\xi_3\in (0,x)$ and $\xi_4\in(0,x)$ for a given $x\in [0,a]$ such that
\begin{equation*}
    \psi(x)=\psi'(0)x+\frac{\psi''(\xi_3)}{2}x^2,
\;\text{and}\;
    \psi'(x)=\psi'(0)+\psi''(\xi_4)x.
\end{equation*}
It is derived by a direct computation that
\begin{equation*}
    |\tilde{\psi}'(x)|=\left| \frac{\psi'(x)x-\psi(x)}{x^2}\right|= \left|\psi''(\xi_3)/2- \psi''(\xi_4)\right| \leq \frac{3}{2}\|\psi''\|_\infty.
\end{equation*}
Combining the discussion above, we get that
\begin{equation}\label{errbound}
   E_n(f)\leq C(1+|\ln(w)|) w^{-2} \left(2\|\psi'\|_\infty+ \frac52 a\|\psi''\|_\infty\right).
\end{equation}
By Rolle's theorem, there exist $z_j\in (y_j, y_{j+1})$ such that
\[
\psi''(z_j)=0,\; j=0,1,\ldots,n-3,
\]
and then we derive from the interpolation errors that
\[
 \psi''(x)=\frac{\psi^{(n)}(\xi_5)}{(n-2)!}\prod_{j=0}^{n-3} (x-z_j),
\]
where $\xi_5\in[0,a]$.
The bound for $E_n(f)$ follows naturally.
\end{proof}

It is obvious that when both end points are included in collocation points in the new Levin method, the asymptotic order is about $\bO(w^{-2}(1+\log|w|))$ according to Theorem \ref{thm1}.
}

We next give the error analysis for Algorithm \ref{alg} in which the collocation points are selected to be the modified Chebyshev-Gauss-Lobatto points. To this  end,
we recall the errors of $f-\hat{f}$ and its derivative where $\hat{f}$ is the interpolant on $[-1,1]$ based on the Chebyshev points. Let $\|\cdot\|_{T}$ be the Chebyshev-weighted 1-norm defined by
\[
\|u\|_T=\int_{-1}^1\frac{|u'(t)|}{\sqrt{1-t^2}}dt
\]
and
denote $\|u\|_{[-1,1],\infty}:=\max_{x\in[-1,1]}|u(x)|$.
{\color{black}
\begin{lemma}[Xiang et. al \cite{XIANG2010chebyshev}]\label{cheb}
    (\rnum{1}) If $f,f',\ldots,f^{(k-1)}$ are absolutely continuous on $[-1,1]$ and if $\|f^{(k)}\|_T=V_k<\infty$ for some $k\geq 1$ and $\hat{f}$ is the interpolant of $f$ of degree $n-1$ based on Chebyshev-Gauss-Lobatto points, then for each $n\geq k+2$,
    \begin{eqnarray}
      \|f-\hat{f}\|_{[-1,1],\infty} & \leq & \frac{4V_k}{k\pi (n-1)(n-2)\ldots (n-k)}, \label{chebinter1}\\
      \|f'-\hat{f}'\|_{[-1,1],\infty} &\leq& \frac{4nV_k}{(k-2)\pi (n-1)(n-3)(n-4)\ldots (n-k)}, k>2, \label{chebinterd1} \\
      \|f''-\hat{f}''\|_{[-1,1],\infty} &\leq& \frac{4n(n-1)V_k}{(3k-4)\pi(n-3)(n-4)(n-5)\ldots (n-k)}, k>4, \label{chebinterd2}
    \end{eqnarray}
    (\rnum{2}) If $f$ is analytic with $|f(z)|\leq M$ in the region bounded by the ellipse with foci $\pm 1$ and major and minor semiaxis lengths summing to $\rho>1$, then there exists a constant $C$ independent of $n$ and $\rho$ such that
        \begin{eqnarray}
       \|f(x)-\hat{f}(x)\|_{[-1,1],\infty} & \leq & C\rho^{-n}, \label{chebinter2}\\
       \|f'(x)-\hat{f}'(x)\|_{[-1,1],\infty} &\leq& Cn^2 \rho^{-n}, \\
       \|f''(x)-\hat{f}''(x)\|_{[-1,1],\infty} &\leq& C n^4 \rho^{-(n-1)}.
    \end{eqnarray}
\end{lemma}

\begin{theorem}\label{thm2}
    (\rnum{1}) Suppose that $f,f',\ldots,f^{(k-1)}$ are absolutely continuous on $[0,a]$ and $\int_{0}^a\frac{|f^{(k)}(t)|}{\sqrt{at-t^2}}dt=V_k<\infty$ for some $k\geq 1$ , then the numerical integral computed by Algorithm \ref{alg}
    satisfies for $n\geq k+2$ and $k>4$,
\begin{equation}\label{cheberr1}
    E_n(f) \leq C(1+|\ln(w)|) w^{-2}\frac{n(n-1)}{(n-3)(n-4)\ldots(n-k)},
\end{equation}
where $C$ is a constant independent of $n$ and $w$.

    (\rnum{2}) If $f((\cdot+1)a/2)$ is analytic with $|f((z+1)a/2)|\leq M$ in the region bounded by the ellipse with foci $\pm 1$ and major and minor semiaxis lengths summing to $\rho>1$, then  the numerical integral computed by Algorithm \ref{alg} satisfies for each $n\geq 1$,
\begin{equation}\label{cheberr2}
    E_n(f) \leq C(1+|\ln(w)|) w^{-2}n^4\rho^{-(n-1)},
\end{equation}
where $C$ is a constant independent of $n$, $\rho$ and $w$.
\end{theorem}

\begin{proof}
Let $F(t):=f((t+1)a/2), t\in[-1,1]$ and the interpolant of $F$ based on Chebyshev-Gauss-Lobatto points $x_j, j=0,1,\ldots,n-1$
is denoted by $\hat{F}$. It is easily obtained that
\begin{equation*}
     \|f'-\hat{f}'\|_\infty=\frac{2}a \|F'-\hat{F}'\|_\infty, \; \text{and}\;  \|f''-\hat{f}''\|_\infty=\frac{4}{a^2} \|F''-\hat{F}''\|_\infty.
\end{equation*}
The error bounds follows directly by combining the results of Lemma \ref{cheb} and the inequality \eqref{errbound} in the proof of Theorem \ref{thm1}.
\end{proof}

We easily conclude from Theorem \ref{thm2} that the new Levin method possesses the quasi-superalgebaric convergence with respect to the number of collocation points for logarithmically singular and oscillatory integrals with a linear oscillator when the Chebshev points are adopted and $f$ is analytic.
}

In the left of this section, we discuss the approximation error of Algorithm \ref{alg2} for singular and oscillatory integrals with general oscillators.
For this purpose, we first present an equivalent algorithm for Algorithm \ref{alg2}.
Let $Q_{w,n}^{[0,a],L}[f,g]$ denote the numerical algorithm proposed in \cite{LI2008} of the classic Levin method for $I_w^{[0,a]}[f,g]$ based on $n$ collocation points.
\begin{lemma}\label{lemma2}
    If $f(x), g(x)$ is suitably smooth and $g'(x)\neq 0$ for $x\in[0,a]$ , then there exists
    \begin{equation}\label{qeq}
        Q_{\log,w,n}^{[0,a],L}[f,g]=Q_{w,n}^{[0,a],L}[f_1,g]+
        \tilde{Q}_{\log,w,n}^{[0,g(a)],L}[f_2]
    \end{equation}
    where $\tilde{Q}_{\log,w,n}^{[0,g(a)],L}[f_2]$ means the modified algorithm \ref{alg} with $\hat{\Bx}=g(\phi(\Bx_T))$, $f_1(x)= f(x)\log\frac{x}{g(x)}$ for $x\neq 0$ and $f_1(0)=f(0)\log\frac{1}{g'(0)}$ and
    $f_2(x)=\frac{f(g^{-1}(x))}{g'(g^{-1}(x)}$.
\end{lemma}
\begin{proof}
Applying a change of variables, $y=g(x)$, to the second part of \eqref{nonlinear}, we get that
    \[
    \int_0^af(x)\log(x)e^{iwg(x)}dx=I_{w}^{[0,a]}[f_1,g]+I_{\log,w}^{[0,g(a)]}[f_2,\ell]
    \]
    where $\ell(x)=x$.
    In Algorithm \ref{alg2}, it adopts the classical Levin method, $Q_{w,n}^{[0,a],L}[f_1,g]$, to approximate $I_{w}^{[0,a]}[f_1,g]$.

    Applying Algorithm  \ref{alg} to $I_{\log,w}^{[0,g(a)]}[f_2,\ell]$, we obtain that
\begin{eqnarray}
q_1'(x)+iw q_1(x) &=& f_2(x) \\
  h_1'(x)+iwh_1(x) &=& -q_2(x)  \\
  h_2'(x)+iwh_2(x) &=& -q_1(0)\frac{1-e^{-iwx}}{x} \label{hh22}
\end{eqnarray}
where $
q_2(x)=\frac{q_1(x)-q_1(0)}{x}, x\neq 0,\;\; \text{and} \; q_2(0)= f_2(0)-iwq_1(0).$ Instead of using the collocation points $\hat{\emph{\Bx}}=\phi(\Bx_T)$, we use the points $\hat{\emph{\Bx}}=g(\phi(\Bx_T))$ and
the above equations can be written as
\begin{eqnarray}
\left[(q_1(g(x)))'_x+iw g'(x)q_1(g(x))\right]_{x=\phi(\Bx_T)} &=& f(\phi(\Bx_T)) \\
\left[(h_1(g(x)))'_x+iw g'(x)h_1(g(x))\right]_{x=\phi(\Bx_T)}  &=& -(q_2(g(x))g'(x))_{x=\phi(\Bx_T)}
 %\\ h_2'(x)+iwh_2(x) &=& -q_1(0)\frac{1-e^{-iwx}}{x}
\end{eqnarray}
where $
q_2(g(x))=\frac{q_1(g(x))-q_1(0)}{g(x)-g(0)}, x\neq 0,\;\; \text{and} \; q_2(0)= \frac{f(0)-iwg'(0)q_1(0)} {g'(0)}.$ The linear system above is as the same as that discretized from equations \eqref{eq24} and \eqref{eq18}. They have the same solutions $\Bq_1$ and $\Bh_1$.
Thus, according to Algorithm \ref{alg}, we have that
\[
\begin{split}
\tilde{Q}_{\log,w,n}^{[0,g(a)],L}[f_2]=
\left(e^{iwg(a)}\Be_n-\Be_1\right)\left(\Bq_1\log(g(a)) + \Bh_1\right)+ e^{iwg(a)} h_2(g(a))
\end{split}
\]
where $h_2$ is the solution of \eqref{hh22} which is given in \eqref{h2}.
Comparing with the expression in Algorithm \ref{alg2}, it follows the equivalence \eqref{qeq}.
\end{proof}

{\color{black} The error analysis for Algorithm \ref{alg2} is listed as a theorem. Let $E_n(f,g):=\left|I_{\log,w}^{[0,a]}[f,g]-Q_{\log,w,n}^{[0,a],L}[f,g]\right|$ denote the absolute error.
\begin{theorem}\label{thm3}
Suppose that $f\in C^{n}[0,a]$ and $g\in C^{n+1}[0,a]$ satisfying that $g(0)=0$, $g'(x)>0, x\in [0,a]$ and the assumptions on $g$ in Lemma \ref{xlemma}, $n\geq 2$. Let $f_1(x)= f(x)\log\frac{x}{g(x)}$ for $x\neq 0$ and $f_1(0)=f(0)\log\frac{1}{g'(0)}$ and
    $f_2(x)=\frac{f(g^{-1}(x))}{g'(g^{-1}(x)}$. Then the numerical integral computed by Algorithm \ref{alg2} based on arbitrary points $\hat{\Bx}=\{0\leq\hat{x}_0<\hat{x}_1<\ldots<\hat{x}_{n-1}\leq a\}$ satisfies
\begin{equation}\label{error}
    E_n(f,g)\leq C\frac{(1+|\ln(w)|) w^{-1} } {(n-1)!}\left( \|f_2^{(n)}\|_\infty (g(a))^n+ \|f_1^{(n)}\|a^n+ \|(\bL\hat{p})^{(n)}\|a^n \right),
\end{equation}
where $\hat{p}$ is the numerical solution of $p'(x)+iwg'(x)p(x)=f_1(x)$ by collocation methods.
Specially, when $\hat{x}_0=0$, $\hat{x}_{n-1}=a$ and $n\geq 3$, i.e. both endpoints are included, there exists
\begin{equation}\label{error2}
    E_n(f,g)\leq C\frac{(1+|\ln(w)|)w^{-2}}{(n-2)!} \left( \|f_2^{(n)}\|_\infty (g(a))^{n-1}+ \|f_1^{(n)}\|a^{n-1}+ \|(\bL\hat{p})^{(n)}\|a^{n-1} \right).
\end{equation}
\end{theorem}
\begin{proof} According to Lemma \ref{lemma2},
  \begin{equation*}
  \begin{split}
    E_n(f,g)
    &=\left|\int_0^a\left(f_1(x)-\bL \hat{p}(x)\right) e^{iwg(x)}dx+\int_0^{g(a)}\left(f_2(x)-\hat{f}_2(x)\right)\log (x) e^{iwx}dx\right| \\
    &\leq \left|\int_0^a\left(f_1(x)-\bL \hat{p}(x)\right) e^{iwg(x)}dx \right| +\left|\int_0^{g(a)}\left(f_2(x)-\hat{f}_2(x)\right)\log (x) e^{iwx}dx\right|,
    \end{split}
  \end{equation*}
  where $\hat{f}_2$ be the interpolant of $f_2$ based on the points $\hat{\emph{\Bx}}$. According to the assumptions, $f_1\in C^n[0,a]$ and $f_2\in C^n[0,g(a)]$.
  Let $\hat{f}_1$ be the interpolant of $f_1$ based on the points $ \hat{\emph{\Bx}}$. Note that $\hat{f}_1$ is also the interpolant of $\bL \hat{p}$. Therefore,
  \begin{equation*}
    \left|\int_0^a\left(f_1(x)-\bL \hat{p}(x)\right) e^{iwg(x)}dx \right|\leq \left|\int_0^a\left(f_1(x)-\hat{f}_1(x)\right) e^{iwg(x)}dx \right|+\left|\int_0^a\left( \bL \hat{p}(x)-\hat{f}_1(x)\right) e^{iwg(x)}dx \right|.
  \end{equation*}
  The desired error bounds are obtained by the similar proof of Theorem \ref{thm1} with Lemma \ref{xlemma}.
\end{proof}

Theorem \ref{thm3} tells that the new Levin algorithm for a general oscillator also possesses the same asymptotic order as that for linear oscillator. Numerical experiments in the later section will show that it also has the quasi-superalgebraic convergence for the general oscillator with respect to the number of collocation points when $f$ is smooth.
}

\section{Numerical experiments}
In this section, we present four numerical experiments to verify the efficiency of the proposed new Levin method for singular oscillatory integrals. To achieve this goal, we also compare the
computational performance of the proposed methods with that of the
quadrature rules proposed in \cite{DOMINGUEZ2014,MAXU2017}.
The numerical results presented below were all obtained by using Matlab 2017b on a laptop that has a Intel(R) Core(TM) i7-6500U CPU with 8GB of Ram memory.

\begin{example} In the first example, we consider the moments for Chebyshev pollynomials
\[
\int_{-1}^1T_m(x)\log \left(x^2\right) e^{iwx}dx
\]
whose is the base for the Filon method in computation of the oscillatory integrals with logarithmic singularities. Their exact values can be computed by the recurrence relations \cite{DOMINGUEZ2014}.
\end{example}
These moments can be rewritten as $2\int_{0}^1T_m(x)\log x  e^{iwx}dx+ 2\int_{0}^1T_m(-x)\log x  e^{i(-w)x}dx$ which can be computed by the new Levin method, i.e. by Algorithm \ref{alg}. Since $T_m$ is a polynomial of degree no more than $m$, the moment can be computed exactly by the new Levin method with $m+1$ collocation nodes in theory. We present in Table \ref{Ta1} the absolute errors of the new Levin method in computing the integrals for $m=[2\; 3\; 4\; 5\; 6]$ and $w=[10\; 10^2\; 10^3\; 10^4]$ with $n=m+1$ collocation nodes.
It is found that the absolute errors of the moments computed by the new Levin method attaches the machine precision for different settings of $m$ and $w$. It validates the prediction numerically and also shows the potential efficiency of the new method in computing the oscillatory integrals with logarithmically singularities.

\begin{table}[htb]
\begin{center}
\small
\caption{Absolute errors of Algorithm \ref{alg} for $\int_{-1}^1T_m(x)\log \left(x^2\right) e^{iwx}dx$}
\label{Ta1}
%\begin{tabular}{c|cccc}
%\hline
\def\temptablewidth{1\textwidth} {\rule{\temptablewidth}{1pt}}
\begin{tabular*}{\temptablewidth}{@{\extracolsep{\fill}} ccccc}
$m$& $w=10$  &  $w=10^2$ & $w=10^3$  &  $w=10^4$ \\
\hline
2&$2.4825e-16$&$2.7756e-17$&$1.9395e-18$&$9.6974e-19$\\
3&$2.8475e-16$&$3.2641e-16$&$1.1458e-17$&$5.1824e-19$\\
4&$3.1402e-16$&$6.9389e-17$&$1.4120e-17$&$1.0842e-18$\\
5&$1.0562e-15$&$6.7761e-17$&$8.2217e-18$&$1.4939e-18$\\
6&$5.5511e-16$&$1.2795e-16$&$2.4533e-18$&$8.7411e-19$\\
%\hline
\end{tabular*}
{\rule{\temptablewidth}{1pt}}
\end{center}
\end{table}

{\color{black}
\begin{example}
 This example is aimed to test the dependence of the absolute error of the proposed Levin algorithms, i.e. Algorithms \ref{alg} and \ref{alg2}, on the number of points $n$ and the frequency $w$ by computing the integrals, respectively,
 \[
 \int_0^1e^{x}\log(x)e^{iwx}dx=\frac{-i}{-i+w}\left(\gamma+\Gamma(0,-1-i w)+\Log(-1-iw) \right),
 \]
 and
 \[
 \begin{split}
 \int_0^1(2x+1) & e^{x^2+x}\log(x)e^{iw(x^2+x)}dx \\
 & =  \frac{-i}{-i+w}\left( \gamma+E_1(-2-2iw)+\Log(-1-iw)+e^{2+2iw}\log2\right) \\ &\quad\quad -\int_0^1(2x+1)e^{x^2+x}\log(x+1)e^{iw(x^2+x)}dx,
 \end{split}
 \]
 where $E_1(z):=\int_1^\infty e^{-zt}{t}dt$ is an exponential integral. Since the classic Levin method for oscillatory integrals without singularities is well developed, we use the classic Levin method with 32 points to form the reference value for the second oscillatory integral with a nonlinear oscillator.
\end{example}

Numerical results of absolute errors are shown in Table \ref{Ta2} and Figure \ref{fig1} for different values of $w$ and $n$. For the dependence on $n$, it is shown that the absolute errors of the new Levin method decay drastically, as $n$ increases slowly for fixed $w=10^2$ and $10^5$ no matter the oscillator is linear or nonlinear. It is consistent with the theory that the error possesses the quasi-superalgebraic convergence. As shown in Figure \ref{fig1}, the absolute errors scaled by $\frac{w^2}{1+\log|w|}$ are bounded for fixed $n$ which validates numerically that the asymptotic order on the frequency is $\bO(w^{-2}(1+\log|w|))$ for both linear and nonlinear cases which matches well with the theoretical results in Theorems \ref{thm1} and \ref{thm2}.
\begin{table}[htb]
\begin{center}
\small
\caption{ Absolute errors for computing $\int_0^1e^{x}\log(x)e^{iwx}dx$ (denoted by Linear) and $ \int_0^1(2x+1)  e^{x^2+x}\log(x)e^{iw(x^2+x)}dx$ (denoted by Nonlinear) by the new Levin method for fixed $w$}
\label{Ta2}
%\begin{tabular}{c|cc|c|cc}
%\hline
\def\temptablewidth{1\textwidth} {\rule{\temptablewidth}{1pt}}
\begin{tabular*}{\temptablewidth}{@{\extracolsep{\fill}} cccccc}
\multirow{2}*{\centering $n$}  & \multicolumn{2}{c}{ Linear }  & \multirow{2}*{\centering $n$} & \multicolumn{2}{c}{Nonlinear} \\
\cline{2-3}\cline{5-6}
& $w=10^2$     & $w=10^5$  &   & $w=10^2$ & $w=10^5$  \\
\hline
6&$1.8700e-08$&$4.7101e-14$&8&$1.5615e-06$&$3.4057e-12$\\
7&$8.0027e-10$&$2.0339e-15$&10&$4.1207e-08$&$8.8854e-14$\\
8&$2.9641e-11$&$7.4714e-17$&12&$8.2915e-10$&$1.7505e-15$\\
9&$9.3690e-13$&$2.3115e-18$&14&$1.4946e-11$&$2.7616e-17$\\
10&$2.6924e-14$&$1.9193e-19$&16&$4.1982e-13$&$3.6692e-19$\\
11&$7.4312e-16$&$9.2478e-20$&18&$2.5710e-14$&$8.1948e-20$\\
%\hline
\end{tabular*}
{\rule{\temptablewidth}{1pt}}
\end{center}
\end{table}
\begin{figure}
  \centering
  \includegraphics[width=4in]{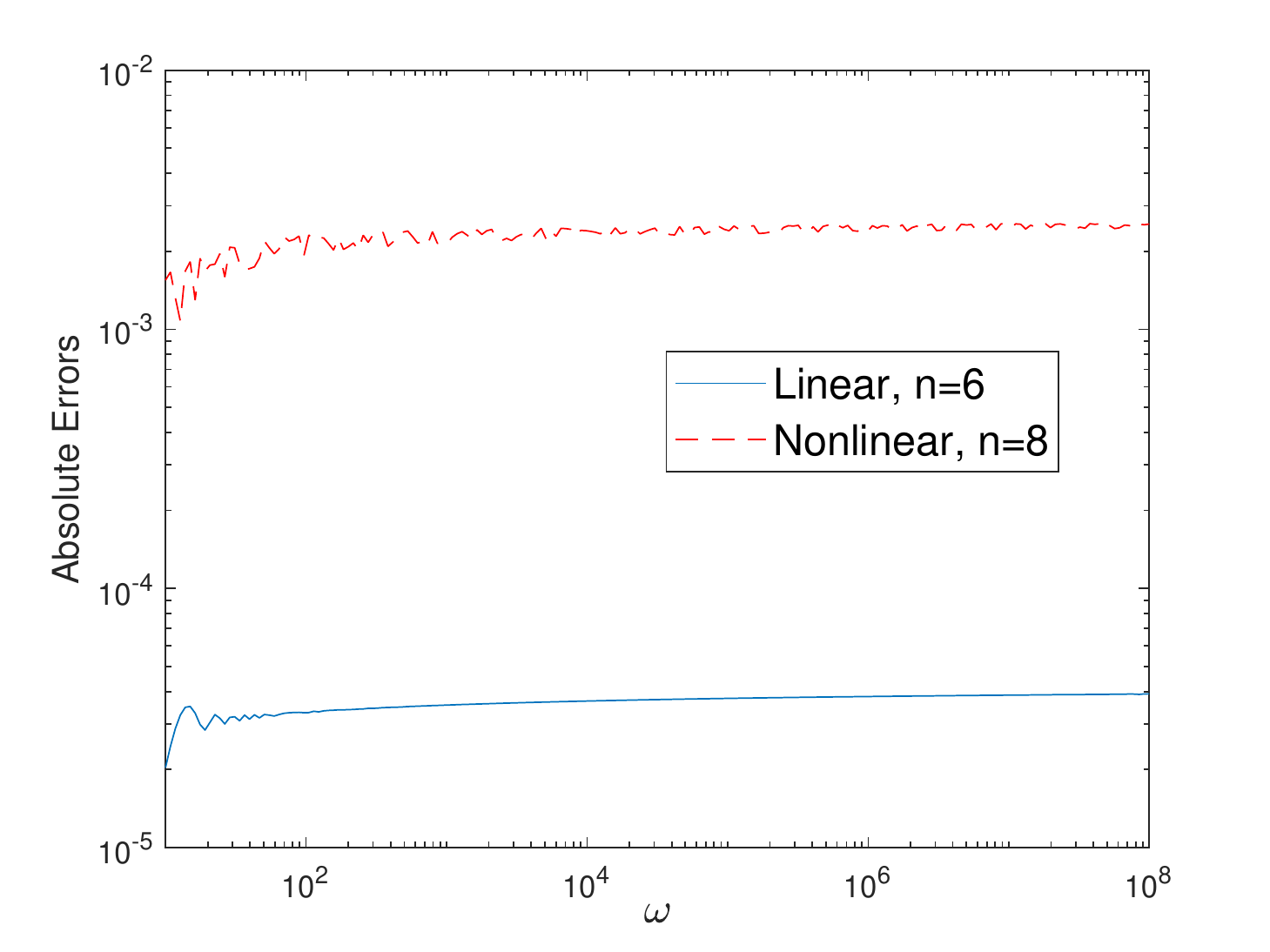}\\
  \caption{ Absolute errors scaled by $\frac{w^2}{1+\log|w|}$ for computing $\int_0^1e^{x}\log(x)e^{iwx}dx$ (denoted by Linear) and $ \int_0^1(2x+1)  e^{x^2+x}\log(x)e^{iw(x^2+x)}dx$ (denoted by Nonlinear) by the new Levin method for fixed $n$ }\label{fig1}
\end{figure}

We next compare the performance of the new Levin method with that
of the existing methods, the Filon-Clenshaw-Curtis method (FCC) in \cite{DOMINGUEZ2014} and the composite moment-free Filon-type quadrature (CMFP) with an polynomial order of convergence in \cite{MAXU2017}. The reason of the choice of CMFP instead of the composite moment-free Filon-type quadrature (CMFE) with an exponential order of convergence is that the CMFE is not stable due to the use of the interpolation of large order in each segment as shown in our numerical experiments which are not presented here.
We then recall the quadrature
formulas of \cite{DOMINGUEZ2014,MAXU2017}. The FCC for the integral $\int_{-1}^1f(x)\log x^2 e^{iwx}dx$ is to compute
\[
Q^{FCC}_{w,n}[f]:=\int_{-1}^1 p_n(x)\log \left(x^2\right) e^{iwx}dx
\]
where $p_n$ is a polynomial of degree $n-1$ which interpolates $f$ at Chebyshev points $\cos\frac{j\pi}{n-1}, j=0,1,\ldots,n-1$.
To introduce the CMFP method, we first simply review the (composite) moment-free Filon-type method \cite{XIANG2007} and the Gauss-Legendre quadrature rule. The moment-free Filon method approximates the integral $\int_a^b f(x)e^{iwg(x)}$ by
\[
Q^{[a,b],MF}_{w,m}[f,g]:=\int_{g(a)}^{g(b)} p_n(x)e^{iwx}dx
\]
where $p_n$ is a polynomial of degree $n-1$ which interpolates $\left[(f/g')\circ g^{-1}\right]$ at $g(t_j),j=0,1,\ldots,m$ and $t_j, j=0,1,\ldots,m$ are a set of distinguish points on $[a,b]$. The composite moment-free Filon-type rules used in CMFP reads
\[
Q^{[a,b],CMF}_{w,n,m}[f,g]:=\sum_{j=1}^{n} Q^{[x_{j-1},x_j],MF}_{w,m}[f,g]\; \text{with}\; x_j=a+\frac{j}{n}(b-a), j=0,1,\ldots,n.
\]
The Gauss-Legendre quadrature rule for integral $\int_a^b f(x)dx$ is given by
\[
Q_m^{[a,b],GL}[f]:=\frac{b-a}{2}\sum_{j=1}^m w_j f\left(\frac{(b-a)t_j+b+a}{2}\right)
\]
where $w_j$ and $t_j$ are the standard weights and points of the Gauss-Legendre rule on the domain $[-1,1]$.
Suppose for a nonnegative integer $r$, the function $g\in C^{r+1}[0,1]$ has a single stationary point at zero and satisfies $g^{(j)}(0)=0$ for $j=1,\ldots,r$ and $g^{(r+1)}(x)\neq 0$ for $x\in[0,1]$. Let $\sigma_r:=\|g^{(r+1)}\|_\infty/(r+1)!$, $w_r:=\max\{k\sigma_r,k\}$, and $\lambda_r:=w_r^{-1/(r+1)}$.   The CMFP method for integral $\int_0^1 f(x)e^{ikg(x)}dx$ is established by
\[
Q_{w,n,s,m_1,m_2}^{CMFP}[f,g]:=\lambda_r\sum_{j=1}^{s-1}Q_{m_1}^{[x_j,x_{j+1}],GL}[\varphi]+ \sum_{j=1}^{n} Q^{[y_{j-1},y_{j}],CMF}_{w,N_j,m_2}[f,g]
\]
where $x_0=0$, $x_j=(j/s)^p, p=(2m+1)/(1+\mu)$, $\mu$ is the index of singularity of $f$, $j=1,2,\ldots,s$, $\varphi(x)=f(\lambda_r x)e^{iwg(\lambda_r x)}$,
$y_j=w_r^{(j-n)/n/(r+1)}, j=0,1,\ldots,n$, $N_j=\lceil q_j^{m/(m-1)}\rceil$, $q_j=\max\{|g'(y_{j-1})|,|g'(y_j)|\}x_{j-1}/g(x_{j-1})$, $j=1,2,\ldots,n$ and $\nu$ is the index of singularity of $\left[(f/g')\circ g^{-1}\right]$. When $f$ has only the logarithmic singularity, the value of $\mu$ is set to be 0.

\begin{example} This example is to validate the efficiency of the new Levin algorithm for a linear oscillator, i.e. Algorithm \ref{alg}, by comparing with the FCC and the CMFP. For this purpose, we consider an integral with a complicate integrand
\[
\int_{-1}^1\frac{\cos(4x)}{x^2+x+1}\log \left(x^2\right) e^{iwx}dx
\]
which is also considered in \cite{DOMINGUEZ2014}. The reference value of the integral is obtained by Mathematica with 50 digits.
\end{example}

\begin{table}[htb]
\begin{center}
\small
\caption{Comparison of relative errors of the new Levin method, the FCC and the CMFP for integral $\int_{-1}^1\frac{\cos(4x)}{x^2+x+1}\log \left(x^2\right) e^{iwx}dx$}
\label{Ta4}
\begin{threeparttable}
%\begin{tabular}{c|ccc|ccc}
%\hline
\def\temptablewidth{1\textwidth} {\rule{\temptablewidth}{1pt}}
\begin{tabular*}{\temptablewidth}{@{\extracolsep{\fill}} cccc|ccc}
\multirow{2}*{\centering $n$}  & \multicolumn{3}{c|}{$w=10^2$}  &  \multicolumn{3}{c}{$w=10^3$} \\
\cline{2-7}
& Levin\tnote{1}   & FCC\tnote{2}  & CMFP\tnote{3}  & Levin  &  FCC & CMFP  \\
\hline
16&$3.5678e-09$&$1.5751e-08$&$1.3552e-06$&$5.8255e-10$&$1.4464e-09$&$1.5382e-06$\\
18&$2.4022e-10$&$1.0089e-10$&$1.4756e-08$&$4.7055e-11$&$8.9554e-12$&$3.2035e-08$\\
20&$2.2788e-11$&$2.5227e-11$&$3.9169e-10$&$3.7534e-12$&$2.4413e-12$&$2.4922e-09$\\
22&$2.0498e-12$&$3.5052e-13$&$5.5345e-12$&$2.9866e-13$&$2.9043e-14$&$9.3250e-10$\\
24&$8.4779e-14$&$4.9937e-14$&$7.5142e-14$&$2.4534e-14$&$3.7966e-15$&$7.3508e-11$\\
26&$3.2518e-15$&$7.4980e-16$&$1.1957e-14$&$3.4336e-15$&$1.5461e-16$&$5.9831e-13$\\
28&$2.7006e-15$&$2.2195e-16$&$1.0252e-14$&$1.4372e-15$&$1.3869e-16$&$6.0480e-14$\\
%\hline
\end{tabular*}
{\rule{\temptablewidth}{1pt}}
\begin{tablenotes}
        \footnotesize
        \item[1] Levin: $Q_{\log,w,n}^{[0,1],L}[f_1]+Q_{\log,-w,n}^{[0,1],L}[f_2]$, where $f_1(x)=2\cos(4x)(x^2+x+1)^{-1}$ and $f_2(x)=f_1(-x)$;
        \item[2] FCC: $Q^{FCC}_{w,2(n-3)}[f_1]$;
        \item[3] CMFP: $Q_{w,n_1,n_1,4,4}^{CMFP}[f_1\log(\cdot)]+Q_{-w,n_1,n_1,4,4}^{CMFP}[f_2\log(\cdot)]$, where $n_1=2^{n/2-3}$.
      \end{tablenotes}
\end{threeparttable}
\end{center}
\end{table}

\begin{table}[htb]
\begin{center}
\small
\caption{Comparison of CPU time of the new Levin method, the FCC and the CMFP for integral $\int_{-1}^1\frac{\cos(4x)}{x^2+x+1}\log \left(x^2\right) e^{iwx}dx$ (The settings for each method are as the same as those in Table \ref{Ta4})}
\label{Ta5}
%\begin{tabular}{c|ccc|ccc}
%\hline
\def\temptablewidth{1\textwidth} {\rule{\temptablewidth}{1pt}}
\begin{tabular*}{\temptablewidth}{@{\extracolsep{\fill}} cccc|ccc}
\multirow{2}*{\centering $n$}  & \multicolumn{3}{c|}{$w=10^2$}  &  \multicolumn{3}{c}{$w=10^3$} \\
\cline{2-7}
& Levin   & FCC  & CMFP  & Levin  &  FCC & CMFP  \\
\hline
16&$2.7733e-02$&$1.6734e-01$&$1.9811e-02$&$2.9702e-02$&$1.5733e-01$&$8.9944e-03$\\
18&$2.6089e-02$&$1.6675e-01$&$1.4163e-02$&$2.5042e-02$&$1.6759e-01$&$1.4238e-02$\\
20&$3.1365e-02$&$1.6556e-01$&$2.6576e-02$&$2.7267e-02$&$1.6112e-01$&$2.5190e-02$\\
22&$2.9330e-02$&$1.6962e-01$&$6.1037e-02$&$3.7453e-02$&$1.5765e-01$&$4.3996e-02$\\
24&$3.6097e-02$&$1.6582e-01$&$1.0179e-01$&$2.6806e-02$&$1.8146e-01$&$8.0632e-02$\\
26&$3.7774e-02$&$1.8236e-01$&$1.8171e-01$&$2.7309e-02$&$1.7427e-01$&$1.6384e-01$\\
28&$3.6684e-02$&$1.6819e-01$&$4.2856e-01$&$2.5199e-02$&$1.4903e-01$&$3.4393e-01$\\
%\hline
\end{tabular*}
{\rule{\temptablewidth}{1pt}}
\end{center}
\end{table}

We present in Table \ref{Ta4} and \ref{Ta5} the relative errors and  the CPU time for different values of $w$ and of $n$ by using different methods, respectively. To illustrate the dependence of each method on $n$, we introduce some notations.
Setting $f_1(x)=2\frac{\cos(4x)}{x^2+x+1}$, $f_2(x)=f_1(-x)$ and $n_1=2^{n/2-3}$, there exists
\[
\int_{-1}^1\frac{\cos(4x)}{x^2+x+1}\log \left(x^2\right) e^{iwx}dx=\int_{0}^1f_1(x)\log \left(x\right) e^{iwx}dx+\int_{0}^1f_2(x)\log \left(x\right) e^{-iwx}dx.
\]
For a given $n$ in Table \ref{Ta4} , the proposed Levin method, FCC and CMFP compute the integral through $Q_{\log,w,n}^{[0,1],L}[f_1]+Q_{\log,-w,n}^{[0,1],L}[f_2]$, $Q^{FCC}_{w,2(n-3)}[f_1]$ and $Q_{w,n_1,n_1,4,4}^{CMFP}[f_1\log(\cdot)]+Q_{-w,n_1,n_1,4,4}^{CMFP}[f_2\log(\cdot)]$, respectively.
The settings for the CMFP is chosen according to those used in \cite{MAXU2017}.
 The results in Table \ref{Ta4} show that the accuracy of the proposed Levin method is comparable with that of the FCC and is better than the CMFP. For the CPU time, it is shown in Table \ref{Ta5} that the proposed Levin method outperforms the other two methods. Hence, the new Levin method is more efficient in computing oscillatory integrals with a linear oscillator.

\begin{example}
This example is to confirm the efficiency of the new Levin algorithm for a nonlinear oscillator, i.e. Algorithm \ref{alg2} by considering an integral
\[
\int_{0}^1\log x e^{\frac{iw}{3}\left(2x+\sin\frac{\pi x}{2}\right)}dx
\]
which is considered in \cite{MAXU2017}. The reference value of the integral is obtained by Mathematica with 50 digits. Since the moments are unknown, the FCC is not applicable in this example and then we compare only with the CMFP.
\end{example}

Numerical results of the relative errors and  the CPU time are shown in Table \ref{Ta6} and \ref{Ta7} for different values of $w$ and of $n$ by using different methods, respectively. Setting $f(x)=1$ and $g(x)=\left(2x+\sin\frac{\pi x}{2}\right)/3$,
the proposed Levin method  and the CMFP are implemented for a given $n$ in Table \ref{Ta6} through $Q_{\log,w,n}^{[0,1],L}[f,g]$ and $Q_{w,n_1,n_1,4,4}^{CMFP}[f\log(\cdot),g]$, respectively, where $n_1=2^{n/2-1}$.
It is shown clearly that the new proposed method is more accurate than the CMFP and cost less computation time.
 Therefore, the new method is also more efficient in dealing with oscillatory integrals with a nonlinear oscillator.
\begin{table}[htb]
\begin{center}
\small
\caption{Comparison of relative errors of the new Levin method and the CMFP for integral $\int_{0}^1\log x e^{\frac{iw}{3}\left(2x+\sin\frac{\pi x}{2}\right)}dx$}
\label{Ta6}
\begin{threeparttable}
%\begin{tabular}{c|cc|cc|cc}
%\hline
\def\temptablewidth{1\textwidth} {\rule{\temptablewidth}{1pt}}
\begin{tabular*}{\temptablewidth}{@{\extracolsep{\fill}} ccc|cc|cc}
\multirow{2}*{\centering $n$}  & \multicolumn{2}{c|}{$w=10^2$}&  \multicolumn{2}{c|}{$w=10^3$}  &  \multicolumn{2}{c}{$w=10^4$} \\
\cline{2-7}
& Levin\tnote{1}   & CMFP\tnote{2}   & Levin & CMFP &  Levin  & CMFP  \\
\hline
12&$8.1378e-10$&$7.7319e-07$&$5.8942e-11$&$3.1475e-07$&$6.8270e-12$&$7.5954e-07$\\
14&$2.7196e-11$&$1.5810e-08$&$2.4841e-12$&$3.7577e-08$&$2.8367e-13$&$3.2036e-08$\\
16&$3.6545e-13$&$1.9381e-10$&$6.7204e-14$&$8.7799e-10$&$7.0379e-15$&$2.0101e-09$\\
18&$1.5204e-14$&$2.8810e-12$&$7.2421e-16$&$2.8126e-10$&$7.7043e-16$&$4.9874e-11$\\
20&$1.1551e-15$&$4.5143e-14$&$1.8171e-15$&$3.8299e-12$&$6.5950e-16$&$4.7109e-12$\\
22&$2.1164e-15$&$3.4286e-15$&$1.1856e-15$&$9.4738e-14$&$2.4220e-15$&$3.0953e-13$\\
24&$5.1056e-15$&$6.0634e-15$&$6.5045e-16$&$2.0395e-14$&$2.9516e-15$&$1.3564e-14$\\
%\hline
\end{tabular*}
{\rule{\temptablewidth}{1pt}}
\begin{tablenotes}
        \footnotesize
        \item[1] Levin: $Q_{\log,w,n}^{[0,1],L}[f,g]$, where $f(x)=1$ and $g(x)=\left(2x+\sin\frac{\pi x}{2}\right)/3$;
        \item[2] CMFP: $Q_{w,n_1,n_1,4,4}^{CMFP}[f\log(\cdot),g]$, where $n_1=2^{n/2-1}$.
      \end{tablenotes}
\end{threeparttable}
\end{center}
\end{table}

\begin{table}[htb]
\begin{center}
\small
\caption{Comparison of CPU time of the new Levin method and the CMFP for integral $\int_{0}^1\log x e^{\frac{iw}{3}\left(2x+\sin\frac{\pi x}{2}\right)}dx$ (The settings for each method are the same as those in Table \ref{Ta6})}
\label{Ta7}
%\begin{tabular}{c|cc|cc|cc}
%\hline
\def\temptablewidth{1\textwidth} {\rule{\temptablewidth}{1pt}}
\begin{tabular*}{\temptablewidth}{@{\extracolsep{\fill}} ccc|cc|cc}
\multirow{2}*{\centering $n$}  & \multicolumn{2}{c|}{$w=10^2$}&  \multicolumn{2}{c|}{$w=10^3$}  &  \multicolumn{2}{c}{$w=10^4$} \\
\cline{2-7}
& Levin   & CMFP    & Levin & CMFP &  Levin  & CMFP  \\
\hline
12&$1.7341e-02$&$2.6221e-02$&$1.5810e-02$&$2.3357e-02$&$1.3076e-02$&$1.9889e-02$\\
14&$1.7615e-02$&$2.6740e-02$&$1.4048e-02$&$2.4913e-02$&$1.4106e-02$&$2.2531e-02$\\
16&$1.4016e-02$&$2.8904e-02$&$1.4972e-02$&$3.1208e-02$&$1.3430e-02$&$2.6690e-02$\\
18&$1.4940e-02$&$4.4752e-02$&$1.4144e-02$&$3.8594e-02$&$1.3570e-02$&$3.4280e-02$\\
20&$1.7483e-02$&$6.8332e-02$&$1.4892e-02$&$5.1830e-02$&$1.3344e-02$&$5.1740e-02$\\
22&$1.6228e-02$&$8.8271e-02$&$1.4012e-02$&$8.9576e-02$&$1.4798e-02$&$9.0808e-02$\\
24&$1.4908e-02$&$1.7009e-01$&$1.8327e-02$&$1.7746e-01$&$1.9131e-02$&$1.7763e-01$\\
%\hline
\end{tabular*}
{\rule{\temptablewidth}{1pt}}
\end{center}
\end{table}

Besides, extra numerical results show that when the number of $n_1$ increases up to $2^{14}$, the error of CMFP in Examples 4.3 and 4.4 starts to increase for $w=10^2$ which is due to the round-off error of tiny meshes while the new method is free of the problem caused by tiny meshes.
}

\section{Conclusions}
We have constructed a numerically stable Levin method for computing highly oscillatory
integrals with logarithmically singularities, which does not require knowledge of the
derivatives of $f$. This method retains the most vital computational property: higher frequency requires less work.
The proposed method possesses the asymptotic order with respect to $w$ of $\bO(w^{-2}(1+\log|w|))$ and the quasi-superalgebraic convergence when $f$ is analytic.
As shown in the algorithms, it only needs the comparable computation cost of the classic Levin method.

In the future, we hope to generalize these results for computing oscillatory integrals
with other singularities and stationary points.

%\section{Acknowledgement}
%This research is partially supported by the Construct Program of the
%Key Laboratory of High Performance Computing and Stochastic Information Processing.

%\bibliographystyle{plain}% ¡°standard¡± styles, plain, unsrt, abbrv and alpha.
%\bibliography{Report_9_2}

\end{CJK}
\end{document}